\documentclass[letterpaper, onecolumn, 10pt, conference]{ieeeconf}  

\IEEEoverridecommandlockouts                              

\usepackage{amsmath}
\usepackage{amssymb}
\usepackage{amsfonts}
\usepackage{graphicx}
\usepackage{lipsum}
\usepackage{epstopdf}
\usepackage{algorithmic}
\usepackage{mathptmx}
\usepackage{datetime}
\usepackage{hyperref}
\usepackage{color}

\newtheorem{lem}{Lemma}
\newtheorem{thm}{Theorem}

\DeclareMathAlphabet{\bit}{OML}{cmm}{b}{it}

\def\fH{\mathfrak{H}}

\def\<{\leqslant}           
\def\>{\geqslant}           

\def\d{\partial}
\def\wh{\widehat}
\def\wt{\widetilde}

\def\Re{\mathrm{Re}}   
\def\Im{\mathrm{Im}}   

\def\col{\mathrm{vec}}   
\def\cH{\mathcal{H}}   

\def\mR{\mathbb{R}}    
\def\mC{\mathbb{C}}    

\def\Tr{\mathrm{Tr}}       
\def\rT{\mathrm{T}}        
\def\cov{\mathbf{cov}}

\def\eps{\epsilon}

\def\rank{\mathrm{rank}}       

\def\diam{\diamond}       
\def\bS{\mathbf{S}}         

\def\bE{\mathbf{E}}    

\def\bH{\mathbf{H}}    

\def\[[[{[\![\![}   
\def\]]]{]\!]\!]}   

\def\bra{{\langle}}
\def\ket{{\rangle}}

\def\Bra{\Big\langle}
\def\Ket{\Big\rangle}

\def\re{\mathrm{e}}        
\def\rd{\mathrm{d}}        

\def\lexp{\mathop{\overleftarrow{\exp}}}

\def\bJ{\mathbf{J}}

\def\x{\times}
\def\ox{\otimes}

\def\fF{\mathfrak{F}}

\def\cZ{\mathcal{Z}}

\def\cI{\mathcal{I}}

\def\cA{\mathcal{A}}

\def\cE{\mathcal{E}}

\def\Ups{\Upsilon}
\def\ups{\upsilon}

\begin{document}
\title{\LARGE \bf 
Pointwise 
and Dynamic Programming  
Control Synthesis for  
Finite-Level Open Quantum 
Memory Systems${}^*$}

\author{Igor G. Vladimirov$^{1}$, \quad Ian R. Petersen$^{1}$, \quad Guodong Shi$^{1,2}$%
\thanks{${}^*$This work is supported by the Australian Research Council grant DP240101494.}
\thanks{$^{1}$School of Engineering, Australian National University, Canberra, ACT,
Australia:
        {\tt igor.g.vladimirov@gmail.com,
        i.r.petersen@gmail.com}.}%
\thanks{$^{2}$Australian Centre for Robotics, University of Sydney,
Camperdown, Sydney, NSW,
Australia:
        {\tt guodong.shi@sydney.edu.au}.}%
}

\maketitle
\thispagestyle{empty}
\pagestyle{plain}

\begin{abstract}                
This paper is concerned with finite-level quantum memory systems for retaining initial dynamic variables  in the presence of external quantum noise. The system variables have an algebraic structure, similar to that of the Pauli matrices, and their Heisenberg picture evolution is governed by a quasilinear quantum stochastic differential equation.  The latter involves a Hamiltonian whose parameters depend affinely on a classical control signal in the form of a deterministic function of time. The memory performance is quantified by a mean-square deviation of quantum system variables of interest from their initial conditions.  We relate this functional to a matrix-valued state of an auxiliary classical control-affine dynamical  system. This leads to a pointwise control design where the control signal minimises the time-derivative of the mean-square deviation with an additional  quadratic penalty on the control. In an  alternative finite-horizon   setting with a   terminal-integral cost functional, we apply dynamic programming and obtain a quadratically nonlinear Hamilton-Jacobi-Bellman equation,  for which a solution is outlined in the form of a recursively computed asymptotic expansion. 
\end{abstract}


\section{INTRODUCTION}

Technological applications of quantum mechanics, including quantum communication, quantum information processing and quantum computation architectures  \cite{NC_2000}, aim to employ its specific features as potential  nonclassical resources. One of them is the reversibility of dynamics of isolated quantum systems, which comes from the unitary evolution postulate \cite{S_1994}.  Such dynamics are determined by the energetics of internal interactions in the system specified by its Hamiltonian. The latter, as a  function of operator-valued quantum system variables, combined with their algebraic properties, gives rise to a particular form of the Heisenberg picture equations of motion for an isolated quantum system. Such a system provides a perfect storage for those quantum variables which commute with the Hamiltonian and thus remain constant over the course of time. This property manifests itself at full capacity in the case of zero Hamiltonian, where the quantum state and all the dynamic variables of the system are conserved quantities. 

However, the complete isolation of a quantum system from its environment is practically impossible. Because of unavoidable interaction with external fields, the system acquires irreversible dynamics and deviates from its initial conditions even in the zero-Hamiltonian case. Quantum noise makes the  ability of the resulting open system to retain the initial quantum information fade away, giving  rise to various decoherence effects \cite{BP_2006,CL_1985,GZ_2004,U_1995}.

The open quantum system dynamics, including the interaction with external fields, can be modelled by using Hudson-Parthasarathy  quantum stochastic differential equations (QSDEs) \cite{HP_1984,P_1992} driven by a noncommutative quantum counterpart of the classical Wiener process. These and other tools of quantum stochastic calculus  allow for  the modelling of  coherent (through direct energy and field-mediated coupling) or measurement-based system interconnections  and control protocols  \cite{GJ_2009,NG_2015,WM_2010,ZJ_2011a}. The energy and coupling parameters of the resulting composite system can be varied in order to improve its performance, for example, in the sense of minimising a relevant cost functional.  The latter is organised as the mean-square deviation of quantum system variables of interest from their initial conditions in quantum memory optimisation problems \cite{VP_2024_ANZCC,VP_2024_EJC,VPS_2025_CDC,VPS_2025_IFAC}. They also employ a related criterion of maximising the decoherence time at which the mean-square deviation achieves a critical threshold value.  This approach applies to continuous variables systems, such as open quantum harmonic oscillators governed by linear QSDEs, and finite-level systems, which are described by quasilinear QSDEs \cite{EMPUJ_2016,VP_2022_SIAM} and are particularly relevant to the modelling of multiqubit systems. For these classes  of open quantum systems subjected to vacuum fields, the first, second and higher-order  multipoint   moments of dynamic variables lend themselves to efficient computation, thus contributing  to the tractability of the above mentioned quantum memory optimisation problems which were concerned with coherent and measurement-based settings.

The present paper considers the optimisation of a finite-level open quantum memory system,  whose dynamic variables have an algebraic structure, similar to that of the Pauli matrices \cite{S_1994}, and are governed by a quasilinear QSDE.  The system 
Hamiltonian is a linear function of the system variables, and its coefficients  
depend affinely 
on a classical control signal in the form of a deterministic function of time.  
The latter makes the control setting under consideration different from the coherent and measurement-based formulations mentioned above. The mean-square deviation of selected quantum system variables from their initial values, which quantifies the memory performance, is represented in terms of a matrix-valued state of an auxiliary classical control-affine \cite{S_1998}  dynamical  system. We use this representation for a pointwise optimal control design where, at every moment of time,  
the control signal minimises the sum of the time-derivative of the mean-square deviation and an additional  quadratic penalty on the control. An  alternative yet related finite-horizon optimal control problem is also discussed which involves minimising the sum of the terminal mean-square  deviation of the memory system variables from their initial values and an integral quadratic penalty on the control signal.  We apply dynamic programming and obtain a quadratically nonlinear Hamilton-Jacobi-Bellman equation (HJBE),  for which a solution is outlined in the form of a recursively computed asymptotic expansion.   
Note that the approach of this paper differs from the feedback stabilisation of quantum states   for multiqubit systems \cite{ASDSMR_2013,LAM_2022} since we quantify the system performance in the Heisenberg picture terms, use a deterministic control pattern and employ different optimisation methods.

The paper is organised as follows.
Section~\ref{sec:mem} specifies the class of finite-level open quantum memory systems with deterministic control being considered. 
Section~\ref{sec:dyn} describes the evolution of the system variables and their first two moments. 
Section~\ref{sec:dev} uses them for computing the mean-square deviation of the system variables of interest from initial conditions. 
Section~\ref{sec:aux} represents this functional in terms of a matrix-valued state of an auxiliary dynamical system. 
Section~\ref{sec:lyap} discusses the pointwise minimisation of the time derivative of the mean-square deviation with a quadratic penalty on the control. 
Section~\ref{sec:opt} applies dynamic programming to finite-horizon optimal control synthesis using the  terminal mean-square deviation and an integral quadratic control penalty. 
Section~\ref{sec:conc} makes concluding remarks.

\section{FINITE-LEVEL OPEN QUANTUM MEMORY SYSTEM}
\label{sec:mem}

We consider an open  quantum system with internal dynamic variables $X_1(t), \ldots, X_n(t)$  which are self-adjoint operators on a Hilbert space $\fH$ specified below. They  depend on time $t\> 0$ and  are assembled into a column-vector $X(t) := (X_k(t))_{1\< k \< n}$  (the time arguments will often be omitted).  The system is intended to be employed as a quantum memory. In its storage phase, it has to retain (at least approximately  over a bounded time interval) the initial system variables. The latter  comprise the vector 
\begin{equation}
\label{X0}
    X_0 := X(0) 
\end{equation}
and are organised as Hermitian matrices acting on a finite-dimensional complex Hilbert space $\fH_0$. 
These  operators are assumed to possess algebraic closedness, which is preserved in time by the system evolution and represented in the vector-matrix form \cite{VP_2022_SIAM,VP_2024_EJC} as 
\begin{equation}
\label{Xalg}
    X(t)X(t)^\rT = \alpha  + \beta\cdot X(t),
    \qquad
    \beta\cdot X 
  := 
  \sum_{\ell=1}^n \beta_\ell X_\ell, 
    \qquad
    t \> 0. 
\end{equation}
This means that their  pairwise products are affine functions of the system variables: $X_jX_k = \alpha_{jk}\cI  + \sum_{\ell = 1}^n \beta_{jk\ell} X_\ell$, where $\cI$ is the identity operator on the space $\fH$.  
Here, $\alpha := (\alpha_{jk})_{1\< j,k\< n}= \alpha^\rT \in \mR^{n\x n}$ is a real symmetric matrix, and  
$\beta:= (\beta_{jk\ell})_{1\< j,k, \ell\< n} \in \mC^{n\x n\x n}$ is a complex array whose sections $\beta_\ell:= (\beta_{jk\ell})_{1\< j,k\< n} = \beta_\ell^* \in \mC^{n\x n}$ are Hermitian matrices, with $(\cdot)^*:= \overline{(\cdot)}{}^\rT$  the complex conjugate transpose.  Accordingly, the matrix $\alpha$ in (\ref{Xalg})  represents the tensor product $\alpha \ox \cI = (\alpha_{jk}\cI)_{1\< j,k\< n}$. In a similar fashion,  $\beta_\ell X_\ell:= \beta_\ell \ox X_\ell = (\beta_{jk\ell}X_\ell)_{1\< j,k\< n}$ is an $(n\x n)$-matrix of scaled  copies of the operator $X_\ell$. The array
\begin{equation}
\label{Theta}
    \Theta 
    : =
    (\theta_{jk\ell})_{1\< j,k,\ell \< n}  
    :=
    \Im \beta 
    \in \mR^{n\x n\x n}  
\end{equation}
has real antisymmetric sections (due to $\beta_1, \ldots, \beta_n$  being Hermitian) 
\begin{equation}
\label{Thetaell}
    \Theta_\ell
    :=
    (\theta_{jk\ell})_{1\< j,k\< n}
    =     
    \Im \beta_\ell
    =
    -\Theta_\ell^\rT
    \in
    \mR^{n\x n}   
\end{equation}
and parameterises the  one-point canonical commutation relations (CCRs) 
\begin{equation}
\label{XXcomm}
    [X,X^\rT]
      :=
    ([X_j,X_k])_{1\< j,k\< n}
    =
    2i \Theta \cdot X
\end{equation}
for the system variables,  where 
$[\xi,\eta]:= \xi\eta -\eta\xi$ is the commutator of linear operators, and $i:= \sqrt{-1}$ is the imaginary unit.  

For example, 
an open  qubit system has  the initial space $\fH_0:= \mC^2$ 
and $n=3$ dynamic variables whose initial conditions are provided by the Pauli matrices \cite{S_1994}
$$    
    X_1(0):=
    {\begin{bmatrix}
      0 & 1\\
      1 & 0
    \end{bmatrix}},
    \qquad 
    X_2(0)
    :=
    -i\bJ,
    \qquad
    X_3(0):=
    {\begin{bmatrix}
      1 & 0\\
      0 & -1
    \end{bmatrix}},
$$
where
\begin{equation}
\label{bJ}
        \bJ
        : =
        {\begin{bmatrix}
        0 & 1\\
        -1 & 0
    \end{bmatrix}}. 
\end{equation}
They have the algebraic structure (\ref{Xalg}) 
with $\alpha = I_3$ and $\beta = i \Theta$, where the CCR array $\Theta \in\{0, \pm1\}^{3\x 3\x 3}$ in  
(\ref{Theta}) consists of the Levi-Civita symbols   $\theta_{jk\ell} = \eps_{jk\ell}$.

In general, with the algebraic structure (\ref{Xalg}) allowing any polynomial of the system variables to be reduced to an affine function of $X$, the Hamiltonian (which specifies the  internal energy of the quantum system 
\cite{S_1994})  is a linear function of the system variables:
\begin{equation}
\label{HE}
    H
    :=
    E^\rT X
    =
    H_* + \sum_{k = 1}^r U_k H_k,     
\end{equation}
where $E$ is an $\mR^n$-valued energy vector. The latter depends on an $\mR^r$-valued external control signal $U:= (U_k)_{1\< k \< r}$ as
\begin{equation}
\label{EU}
  E := E_* + K U , 
\end{equation}
so that $U$ plays the role of a time-varying classical  parameter of the system Hamiltonian $H$. 
Here, $E_*\in \mR^n$ and $K\in \mR^{n \x r}$ are given constant vector and matrix. 
Accordingly, $H_*$ in (\ref{HE}) is the uncontrolled part of $H$ (which the system would have in the case of $U=0$), while $H_1, \ldots, H_r$ are control Hamiltonians: 
\begin{equation}
\label{HHH}
  H_* := E_*^\rT X,
  \qquad
  H_k := K_k^\rT X,
  \qquad
  k = 1, \ldots, r,  
\end{equation}
where $K_1, \ldots, K_r \in \mR^n$ are the columns of the  matrix $K$ from (\ref{EU}). For what follows, the control signal $U$ is assumed  to be a sufficiently smooth deterministic function of time (it will be found as a solution of optimisation problems in Sections~\ref{sec:lyap} and \ref{sec:opt}).  Note that, in the case of  isolated quantum systems, the affine dependence of the Hamiltonian on deterministic control,  such as in (\ref{HE}),   leads to a bilinear Schr\"{o}dinger equation and related controllability problems (see, for example,  \cite{LBS_2025} and references therein). However, we are concerned with open quantum dynamics.

More precisely, in addition to the internal dynamics being subject to deterministic control, the system  is also affected by external bosonic fields.   
The latter are modelled by a multichannel quantum Wiener process $W:= (W_k)_{1\< k \< m}$ consisting of an even number $m$ of time-varying self-adjoint operators $W_1, \ldots, W_m$  on a symmetric Fock space \cite{PS_1972} $\fF$. The quantum processes $W_1, \ldots, W_m$  have a complex positive semi-definite Hermitian Ito matrix  $\Omega = \Omega^* \succcurlyeq 0$:
\begin{equation}
\label{Omega}
    \rd W \rd W^\rT
    = \Omega \rd t,
    \qquad
    \Omega: = I_m + iJ,
    \qquad
    J
    :=
    I_{m/2}\ox\bJ,
\end{equation}
where $I_s$ denotes the identity matrix of order $s$, and $\bJ$ is given by (\ref{bJ}). 
The nonzero matrix $J= \Im \Omega$ in (\ref{Omega}) comes from the noncommutative nature of $W$  and parameterises the two-point CCRs 
\begin{equation}
\label{WWcomm}
    [W(s), W(t)^\rT]
    =
    2i\min(s,t) J,
    \qquad
    s,t\>0.
\end{equation}
The initial system space $\fH_0$ and the Fock space $\fF$, mentioned above,  give rise to the tensor-product Hilbert space 
\begin{equation}
\label{fH}
    \fH := \fH_0 \ox \fF 
\end{equation}
for the action of the system and field operators. The system-field quantum state on $\fH$ is assumed to have the product structure 
\begin{equation}
\label{rho}
    \rho:= \rho_0\ox \ups,
\end{equation}
where $\rho_0$ is the initial system state on $\fH_0$, and $\ups$ is the vacuum field state \cite{P_1992} on $\fF$, which are positive semi-definite self-adjoint unit-trace density operators on $\fH_0$ and $\fF$, respectively.

In addition to the Hamiltonian (\ref{HE}) as the self-energy of the system, the energetics of the system-field interaction is specified by a vector
\begin{equation}
  \label{LMN}
    L:= MX + N
\end{equation}
of $m$ self-adjoint coupling operators,  which are affine functions of the system variables parameterised by $M \in \mR^{m\x n}$, $N\in \mR^m$.   

The Heisenberg picture dynamics of the open quantum system (see Fig.~\ref{fig:open})  
\begin{figure}[htbp]
\centering
\unitlength=1.2mm
\linethickness{0.5pt}
\begin{picture}(40.00,10)
    \put(15,0){\framebox(10,10)[cc]{\scriptsize memory}}
    \put(35,5){\vector(-1,0){10}}    
    \put(5,5){\vector(1,0){10}}    
    \put(2,5){\makebox(0,0)[cc]{$U$}}
    \put(38,5){\makebox(0,0)[cc]{$W$}}
\end{picture}\vskip-1mm
\caption{An open quantum memory system governed by (\ref{dX}), which is driven by the external quantum noise $W$ and affected by the control signal $U$.}
\label{fig:open}
\end{figure}
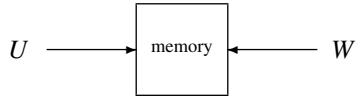
are governed by a quasilinear Hudson-Parthasarathy QSDE  \cite{EMPUJ_2016,VP_2022_SIAM}:
\begin{equation}
\label{dX}
    \rd X  = (AX+b) \rd t + B(X)\rd W. 
\end{equation}
Here, the matrix $A\in \mR^{n\x n}$ is a deterministic function of time $t\> 0$, while $b \in \mR^n$ is a constant vector,  and $B(X)$ is an $(n\x m)$-matrix  of self-adjoint operators depending linearly on the system variables as specified below.   
By an appropriate adaptation of \cite[Theorem 3.1]{VP_2022_SIAM} and \cite[Lemma~4.2 and Theorem~6.1]{EMPUJ_2016}, the coefficients    of the QSDE (\ref{dX}) are computed in terms of the algebraic structure constants $\alpha$, $\beta$  from (\ref{Xalg}), the CCR matrices $\Theta_1, \ldots, \Theta_n$ and $J$ in (\ref{Theta})--(\ref{XXcomm}),  (\ref{Omega}), (\ref{WWcomm}),  and also the energy and coupling parameters $E$, $M$, $N$ from (\ref{HE})--(\ref{HHH}), (\ref{LMN}) as 
\begin{align}
\label{A_b}
    A
    & :=
    A_* + \wt{A},
    \qquad 
    b
      :=
      2
      \begin{bmatrix}
        \Theta_1\   
        \ldots \  
        \Theta_n
    \end{bmatrix}
    \col(M^\rT J M \alpha),\\
\label{BX}
    B(X)
     & := 2(\Theta \cdot X)M^\rT. 
\end{align}
Here, $\col(\cdot)$ is the vectorization \cite{M_1988},    and use is made of auxiliary matrices 
\begin{align}
\label{A*}
    A_*
    & :=    
        2
        \Theta \diam (E_* + M^\rT JN)
        +
    2
    \sum_{\ell = 1}^n
    \Theta_\ell
    M^\rT
    (
        M\theta_{\ell\bullet \bullet}
        +
        J M\Re \beta_{\ell\bullet \bullet}),\\
\label{At}
    \wt{A}
    & :=   
    2\Theta \diam (KU)
    =  
    \sum_{k=1}^r
    U_k A_k = \cA \cdot U, 
\end{align}
where $\cA \in \mR^{n\x n \x r}$ is an array with sections $A_1, \ldots, A_r \in \mR^{n\x n}$ given by
\begin{equation}
\label{Ak}
    A_k
     :=    
    2\Theta \diam K_k 
    =
    2
    \begin{bmatrix}
      \Theta_1 K_k & 
      \ldots & 
      \Theta_n K_k
    \end{bmatrix},
    \qquad
    k = 1, \ldots, r.   
\end{equation}
The relations (\ref{A*})--(\ref{Ak}) employ a different product $\diam$ defined by 
\begin{equation}
\label{diam}
    \Theta \diam u
    :=
    \begin{bmatrix}
        \Theta_1 u & \ldots & \Theta_n u
    \end{bmatrix}. 
\end{equation}
As mentioned in \cite{VP_2022_SIAM}, it is related to the product $\cdot$ in (\ref{Xalg}), (\ref{XXcomm}) by 
\begin{equation}
\label{cdotdiam}
  (\Theta\cdot u)v
  =
  (\Theta\diam v)u
  =
  \begin{bmatrix}
    \Theta_1 & \ldots &  \Theta_n
  \end{bmatrix}
  (u\ox v)
\end{equation}
(with $\ox$ the Kronecker product)
for any vectors $u$, $v$  of $n$ quantum variables with zero cross-commutations: $[u,v^\rT] = 0$. The latter includes the case when at least one of $u$ or $v$ is a $\mC^n$-valued vector. While the matrix $A_*$ in (\ref{A*}) is constant, the matrix $A = A(t)$ in (\ref{A_b}) acquires dependence on time $t$ from the control signal $U$ through the matrix $\wt{A} = \wt{A}(t)$ in  (\ref{At})  and reduces to $A_*$ when $U=0$ since $\wt{A}$ depends on $U$ in a linear fashion. Note that $U$ enters the system dynamics (\ref{dX}) through the matrix $\wt{A}$ only.

\section{EVOLUTION OF SYSTEM VARIABLES AND THEIR MOMENTS}
\label{sec:dyn}

By a variation-of-constants argument as in \cite[Theorem 3.2]{VP_2022_SIAM} (with the only modification that the time dependence of the matrix $A$ is taken into account), the solution of the quasilinear QSDE (\ref{dX})   for the system variables with the initial condition (\ref{X0}) is represented as
\begin{equation}
\label{Xts}
    X(t) = \cE(t,0)X_0 + \int_0^t \cE(t,s)\rd s b,
    \qquad
    t \> 0. 
\end{equation}
Here, $\cE(t,s)
  :=
  (\cE_{jk}(t,s))_{1\< j,k\< n}
$ is an $(n\x n)$-matrix of self-adjoint operators on the Fock space $\fF$, which is  defined for any $t\> s\> 0$ as the leftwards time-ordered operator exponential
\begin{equation}
\label{Ets}
  \cE(t,s)
  :=
    \lexp
  \int_s^t
  (A(\tau)\rd \tau + 2\Theta \diam (M^\rT \rd W(\tau))), 
\end{equation}
and hence, $[\cE_{\bullet k}(t,s), X(s)^\rT] = 0$ for all $k = 1, \ldots, n$.   
The columns $\cE_{\bullet k}(t,s):= (\cE_{jk}(t,s))_{1\< j \< n}$ of the matrix $\cE(t,s)$ satisfy the  common QSDE 
\begin{align}
\nonumber
    \rd_t \cE_{\bullet k}(t,s)
    & =
    A(t) \cE_{\bullet k}(t,s)\rd t + B(\cE_{\bullet k}(t,s))\rd W(t)\\
\label{dEts}
    & =
    (A(t)\rd t + 2\Theta \diam (M^\rT \rd W(t)))\cE_{\bullet k}(t,s)
\end{align}
(see \cite[Eq. (3.33)]{VP_2022_SIAM} or \cite[Eq. (28)]{VP_2024_EJC}), 
with the initial condition $\cE(s,s) = I_n \ox \cI$ which is identified with $I_n$.   By linearity of the map $B$ in (\ref{BX}), the QSDE (\ref{dEts}) is  a homogeneous version of (\ref{dX}), with $b$ removed, and the second equality in (\ref{dEts}) employs the identity (\ref{cdotdiam}) for the product (\ref{diam}).

In the case of vacuum input fields being considered in accordance with (\ref{rho}), the diffusion term $B(\cE_{\bullet k})\rd W$ is a martingale part  of the QSDE (\ref{dEts}). Hence, this term does not contribute to the time derivative of the expectation $\bE \zeta := \Tr (\rho \zeta)$ of a quantum variable $\zeta$ on the system-field space (\ref{fH}) when $\bE(\cdot)$ is evaluated at the entries of the matrix (\ref{Ets}), which leads to
\begin{equation}
\label{EEts}
    \d_t \bE \cE(t,s) 
    = 
    A(t)  \bE \cE(t,s),
    \qquad
    t \> s \> 0
\end{equation}
(cf. \cite[Eq. (4.20)]{VP_2022_SIAM}), with $\bE \cE(s,s) = I_n$. Therefore, the expectation in (\ref{EEts}) coincides with the fundamental matrix $G(t,s) \in \mR^{n\x n}$ of the ODE $\dot{z}(t) = A(t) z(t)$, where $\dot{(\ )}$ is the time derivative:
\begin{equation}
\label{EEtsG}
    \bE \cE(t,s) 
    = 
    G(t,s),
    \qquad
    t \> s \> 0. 
\end{equation}
The mean values of the system variables are obtained by taking the quantum expectation on both sides of (\ref{Xts}) and using (\ref{EEtsG}) as
\begin{equation}
\label{mu}
  \mu(t)
  :=
  \bE X(t) = G(t,0) \mu_0 + \psi(t) b, 
  \qquad
  \mu_0 := \mu(0). 
\end{equation}
Here, $\psi: \mR_+ \to \mR^{n\x n}$ is an auxiliary function  of time defined by 
\begin{equation}
\label{psi}
    \psi(t)
    := 
    \int_0^t 
    G(t,s)\rd s, 
    \qquad
    t \> 0, 
\end{equation}
and thus satisfying the ODE $\dot{\psi}(t)  = A(t) \psi(t) + I_n$ initialised at $\psi(0) = 0$.  
The solution of the QSDE (\ref{dX}) admits an alternative representation in terms of the fundamental matrix $G$ as 
\begin{equation}
\label{Xsol}
    X(t) = G(t,0)X_0 + \psi(t) b + Z(t). 
\end{equation}
Here, (\ref{psi}) is used along with a zero-mean quantum process 
\begin{equation}
\label{Z}
    Z(t)
    :=
    \int_0^t
    G(t,s)
    B(X(s))
    \rd W(s)
\end{equation}
which, while depending functionally  on $X_0$,  is uncorrelated with the initial system variables:
\begin{equation}
\label{EZX0}
    \bE Z(t) = 0,
    \qquad
  \bE(Z(t)X_0^\rT) = 0.
\end{equation}
The algebraic structure (\ref{Xalg}) of the system variables and the statistical properties of the process (\ref{Z}) lead to the following moment relations.

\begin{lem}
\label{lem:EXX}
At any time $t\> 0$, the quantum system variables in (\ref{dX}) satisfy 
\begin{align}
\label{EXXt}
  \bE (X(t) X(t)^\rT) 
  & =  
  \alpha + \beta \cdot \mu(t),\\
\label{EXXt0}
  \bE(X(t)X_0^\rT)
  & = 
    G(t,0) (\alpha + \beta \cdot \mu_0)  + \psi(t)b \mu_0^\rT,
\end{align}
where $\alpha$, $\beta$ are the structure constants from (\ref{Xalg}), the function $\psi$ is associated by (\ref{psi})  with the fundamental matrix $G$ of the ODE (\ref{EEts}), and $\mu(t)$ is the mean vector from (\ref{mu}). \hfill$\square$
\end{lem}
\begin{proof}
The relation (\ref{EXXt}) is obtained by taking the quantum expectation on both sides of the first equality in (\ref{Xalg}). By combining (\ref{Xsol}) with the second equality from (\ref{EZX0}) and using (\ref{EXXt}), it follows that
\begin{align}
\nonumber
  \bE(X(t)X_0^\rT)
  & = 
  G(t,0) \bE (X_0 X_0^\rT)  + \psi(t)b \bE X_0^\rT + \bE(Z(t)X_0^\rT) \\
\label{EXXt0der}
    & = 
    G(t,0) (\alpha + \beta \cdot \mu_0)  + \psi(t)b \mu_0^\rT,
\end{align}
which establishes (\ref{EXXt0}). 
\end{proof}

Note that the two-point second-moment relation (\ref{EXXt0}),  which complements its one-point counterpart (\ref{EXXt}), is an adaptation of \cite[Eq. (4.28)]{VP_2022_SIAM} to the case of time-varying system dynamics.

\section{MEAN-SQUARE DEVIATION FROM INITIAL CONDITIONS}
\label{sec:dev}

We aim the memory control design to preserve the initial condition 
\begin{equation}
\label{varphi0}
  \varphi_0 := \varphi(0)
\end{equation}
of a quantum process
\begin{equation}
\label{FX}
    \varphi(t):= FX(t),
    \qquad
    t \> 0.  
\end{equation}
The latter consists of $\nu\< n$ independent linear combinations (for example, a subset) of the system variables $X_1, \ldots, X_n$. Their coefficients comprise the rows of a given full row rank matrix $F \in \mR^{\nu \x n}$. Accordingly, the control objective is to minimise the deviation 
\begin{equation}
\label{eta}
    \eta(t):= \varphi(t) - \varphi_0 = F \xi(t)  
\end{equation}
in some sense. 
It is related to a similar deviation process $\xi$ for the system variables in (\ref{Xsol}):
\begin{equation}
\label{xi}
  \xi(t)
   :=
  X(t)-X_0
  =
  D(t)X_0 + \psi(t)b + Z(t),  
\end{equation}
where 
\begin{equation}
\label{D}
  D(t):= G(t,0)-I_n. 
\end{equation}
From (\ref{mu}) and (\ref{D}), or alternatively, by the first equality in (\ref{EZX0}), it follows that the mean value of the deviation (\ref{xi}) takes the form 
\begin{equation}
\label{Exi}
  \bE \xi(t) = \mu(t)-\mu_0 = D(t)\mu_0 + \psi(t)b. 
\end{equation}
However, the minimisation of (\ref{eta}) is formulated using a mean-square functional \cite{VP_2024_ANZCC,VP_2024_EJC}
\begin{equation}
\label{Del}
    \Delta(t) := \bE (\eta(t)^\rT\eta(t)) 
    = \bra \Sigma, \Xi(t)\ket , 
\end{equation}
which involves a real positive semi-definite symmetric matrix
\begin{equation}
\label{FF}
    \Sigma := F^\rT F
\end{equation}
of order $n$ with $\rank \Sigma = \nu$. Here, $\bra\cdot, \cdot\ket $    is the Frobenius inner product of matrices \cite{HJ_2007},  and 
\begin{equation}
\label{Xi}
  \Xi(t):= \Re \Ups(t),
  \qquad
    \Ups(t)
    :=
    \bE (\xi(t)\xi(t)^\rT). 
\end{equation}
For what follows, we associate with $\beta$ from (\ref{Xalg}) a real array
\begin{equation}
\label{Rebet}
    \gamma
    := 
    (\gamma_{jk\ell})_{1\< j,k,\ell\< n} 
    := 
    \Re \beta  \in \mR^{n \x n \x n}. 
\end{equation}
Its sections $\gamma_\ell := (\gamma_{jk\ell})_{1\< j,k\< n} = \Re \beta_\ell$ are symmetric matrices since $\beta_\ell$ are Hermitian. In accordance with (\ref{EXXt}), the array $\gamma$ also plays a role in the second moments of the system variables, including
\begin{equation}
\label{ReEXX0}
    P:= \Re \bE(X_0 X_0^\rT) = \alpha + \gamma \cdot \mu_0. 
\end{equation}
This matrix is positive semi-definite and, moreover, $P \succcurlyeq \mu_0\mu_0^\rT$ since $P- \mu_0\mu_0^\rT = \Re \cov(X_0)$ is the real part  of the quantum covariance matrix 
\begin{align}
\nonumber
    \cov(X_0)
    & := 
    \bE ((X_0-\mu_0)(X_0-\mu_0)^\rT)\\
\label{covX0} 
    & = 
    \alpha + \beta \cdot \mu_0 - \mu_0 \mu_0^\rT\succcurlyeq 0,  
\end{align}
with $\Im \cov(X_0) = \frac{1}{2i} \bE [X_0, X_0^\rT] = \Theta \cdot \mu_0$ in view of (\ref{XXcomm}).  

\begin{lem}
\label{lem:ReExixi}
For any time $t\> 0$,  the matrix $\Xi(t)$ in (\ref{Xi}) can be computed as 
\begin{equation}
\label{Xi1} 
    \Xi(t)
    =  
    \gamma \cdot (\mu(t)-\mu_0)- Q(t) - Q(t)^\rT ,
\end{equation}
in terms of (\ref{Exi}) and  
\begin{equation}
\label{ReExiX}
    Q(t)
    := 
    \Re \bE(\xi(t)X_0^\rT) = D(t)P  + \psi(t)b \mu_0^\rT
\end{equation}
using $\gamma$, $P$ from (\ref{Rebet}), (\ref{ReEXX0}),  the functions $\psi$, $D$ from (\ref{psi}), (\ref{D}), and the vectors $b$, $\mu_0$ from (\ref{A_b}), (\ref{mu}).  
\hfill$\square$
\end{lem}

\begin{proof}
From the definition of $\xi$ in (\ref{xi}), it follows that $X(t) = X_0 + \xi(t)$,  and hence, 
\begin{align}
\nonumber
    \xi(t) \xi(t)^\rT 
    & = 
    X(t)X(t)^\rT - X_0 X_0^\rT - \xi(t) X_0^\rT - X_0 \xi(t)^\rT  \\
\label{xixi}
    & = \beta \cdot \xi(t) - \xi(t) X_0^\rT - X_0 \xi(t)^\rT, 
\end{align}
where the last equality uses the algebraic structure (\ref{Xalg}) of the system variables. Application of the quantum expectation to both sides of (\ref{xixi}) allows the second-moment matrix of $\xi(t)$ in (\ref{Xi}) to be computed as
\begin{align}
\nonumber
    \Ups(t)
    & = \beta \cdot \bE \xi(t) - \bE(\xi(t) X_0^\rT + X_0 \xi(t)^\rT)\\
\label{Ups}
    & = 
    \beta \cdot (\mu(t)-\mu_0) - 2\bH(\bE(\xi(t) X_0^\rT)).  
\end{align}
Here, use is also made of (\ref{Exi}), the Hermitian part  $\bH(z):= \frac{1}{2} (z+z^*)$ of a complex square matrix $z$ and the fact that $\bE (uv^\rT) = (\bE (vu^\rT))^*$ for any vectors $u$, $v$ consisting of self-adjoint operators. By a combination of (\ref{xi})  with the second equality in (\ref{EZX0}), 
\begin{align}
\nonumber
    \bE(\xi(t)X_0^\rT) 
    & = D(t) \bE (X_0X_0^\rT) + \psi(t)b \bE X_0^\rT + \bE (Z(t)X_0^\rT) \\
\label{ExiX0}
    &= D(t)\bE (X_0X_0^\rT)  + \psi(t)b \mu_0^\rT, 
\end{align}
which can also be obtained from (\ref{EXXt0der}), (\ref{D}).  In view of (\ref{ReEXX0}), the real part of (\ref{ExiX0}) yields the matrix $Q(t)$ in (\ref{ReExiX}). By substituting  (\ref{ExiX0}) into (\ref{Ups}) and using the identity $\Re \bH(z) = \bS(\Re z)$ for any complex square matrix $z$ along with the symmetrization $\bS(w):= \frac{1}{2}(w+w^\rT)$ for real square matrices $w$, it follows that $\Xi(t)$ in (\ref{Xi}) takes the form
\begin{align*}
    \Xi(t)
    & = 
    \gamma \cdot (\mu(t)-\mu_0) - 2\Re \bH(\bE(\xi(t) X_0^\rT))\\  
    & =  
    \gamma \cdot (\mu(t)-\mu_0) - 2\bS(Q(t)), 
\end{align*}
with $\gamma$, $Q(t)$ given by (\ref{Rebet}), (\ref{ReExiX}), thus establishing the relation (\ref{Xi1}). 
\end{proof}

Note that in view of (\ref{xi}), the matrices (\ref{ReEXX0}), (\ref{ReExiX}) are related to the matrix (\ref{EXXt0}) as
\begin{equation}
\label{QPsum}
    P + Q(t) = \Re \bE ((X_0+\xi(t))X_0^\rT)=\Re \bE (X(t)X_0^\rT). 
\end{equation}

\section{AUXILIARY CONTROL-AFFINE DYNAMICAL SYSTEM}
\label{sec:aux}

For what follows, we associate with the mean vector (\ref{mu}) and the matrices (\ref{ReEXX0}), (\ref{ReExiX}) 
a real time-varying $(n\x (n+1))$-matrix 
\begin{equation}
\label{muQP}
    z(t) := 
    \begin{bmatrix}
      \mu(t) & P + Q(t)
    \end{bmatrix},
    \qquad
    t \> 0. 
\end{equation}
The quantum probabilistic meaning of its rightmost $(n\x n)$-block is provided by (\ref{QPsum}). Furthermore, since $Q_0 := Q(0) = 0$ in (\ref{ReExiX}), then 
\begin{equation}
\label{z0}
    z_0
    := 
    z(0) 
    = 
    \begin{bmatrix}
      \mu_0 & P
    \end{bmatrix}. 
\end{equation}
The matrix $z_0$ cannot be arbitrary and is contained by a bounded set
\begin{equation}
\label{cZ}
  \cZ:= 
  \{
  \begin{bmatrix}
    v & \alpha + \gamma \cdot v
  \end{bmatrix}:\ 
  v \in \mR^n,\ 
  \alpha + \beta \cdot v \succcurlyeq vv^\rT
  \}
\end{equation}
in view of (\ref{ReEXX0}), (\ref{covX0}). In turn,  the set $\cZ$ is contained by the intersection of an ellipsoid with an $n$-dimensional affine subspace of $\mR^{n \x (n+1)}$ (see also \cite[Eq. (2.22)]{VP_2022_SIAM}). 
The significance of (\ref{muQP}) for the underlying quantum memory system is clarified by Theorem~\ref{th:Delz} below. Its formulation employs  the following constant real matrix, vector and scalar:  
\begin{align}
\label{R}
    R & := 
    \begin{bmatrix}
      \sigma & -2\Sigma
    \end{bmatrix}, \\
\label{sig}
    \sigma 
    & := 
    (\bra\Sigma, \gamma_\ell\ket)_{1\< \ell \< n} 
    =
    \Sigma \star \gamma,\\
\label{d}
    d 
    & := 
    2\bra \Sigma, P\ket-\sigma^\rT \mu_0 = -\bra R, z_0\ket,  
\end{align}
which are computed in terms of the matrices $\Sigma$, $P$ from (\ref{FF}), (\ref{ReEXX0}), the sections $\gamma_1, \ldots, \gamma_n$ of the array (\ref{Rebet}), the vector $\mu_0$ in (\ref{mu}), and related to the matrix $z_0$ from  (\ref{z0}). Here, for any matrix $u \in \mR^{p\x q}$ and any array $v \in \mR^{p\x q\x s}$ with sections $v_1, \ldots, v_s \in \mR^{p\x q}$, we use a product $u\star v := (\bra u,v_k\ket )_{1\< k  \< s} \in \mR^s$ satisfying the identity 
\begin{equation}
\label{ustarv}
    \bra u, v \cdot w \ket  
    =
    \sum_{k=1}^s
    \bra u,v_k\ket
    w_k
    =
    (u\star v)^\rT w
\end{equation}
for any vector $w \in \mR^s$.

\begin{thm}
\label{th:Delz}
For any time $t\> 0$, the mean-square deviation (\ref{Del}) is an affine function 
\begin{equation}
\label{DelRzd0}
  \Delta(t)
    =
    \bra R, z(t)\ket + d   
    =
    \bra R, z(t)-z_0\ket
\end{equation}
(with the constant coefficients in (\ref{R})--(\ref{d})) of the matrix (\ref{muQP}). The latter  satisfies a nonhomogeneous linear ODE
\begin{equation}
\label{zdot}
    \dot{z}(t) = A(t) z(t) + c,
    \qquad
    c := 
    b\begin{bmatrix}
      1 & \mu_0^\rT
    \end{bmatrix}
\end{equation}
with the time-varying matrix $A$ from  (\ref{A_b}),  the constant forcing term $c \in \mR^{n \x (n+1)}$ and the initial condition (\ref{z0}). \hfill$\square$
\end{thm}
\begin{proof}
The representation (\ref{DelRzd0}) is obtained by substituting (\ref{Xi1}) into the right-hand side of (\ref{Del}) and using (\ref{muQP})--(\ref{ustarv}):
\begin{align*}
  \Delta(t)
  & = 
  \bra
    \Sigma,
    \gamma \cdot (\mu(t)-\mu_0)- Q(t) - Q(t)^\rT
  \ket\\
    & = 
    (\Sigma \star \gamma)^\rT 
    (\mu(t)-\mu_0) - 2\bra \Sigma, Q(t)\ket \\
    & = 
    \sigma^\rT \mu(t)    
    -
    2\bra \Sigma, P+Q(t)\ket + d
    =
    \bra R, z(t)\ket + d, 
\end{align*}
where $\bra \Sigma, Q^\rT\ket = \bra \Sigma, Q\ket$ due to  the symmetry of the matrix $\Sigma$ in  (\ref{FF}). The second equality in (\ref{DelRzd0}) follows from the last equality in (\ref{d}) and corresponds to the zero initial deviation in (\ref{Del}):
\begin{equation}
\label{Del0}
    \Delta_0 := \Delta(0) = 0. 
\end{equation}
Now, the time-varying matrix (\ref{ReExiX}) satisfies a nonhomogeneous linear ODE 
\begin{align}
\nonumber   
    \dot{Q}(t)
    & = 
    \dot{D}(t) P +\dot{\psi}(t) b \mu_0^\rT\\
\nonumber
     & = \dot{G}(t,0) P +(A(t)\psi(t) + I_n)b \mu_0^\rT\\
\nonumber   
    & = A(t)(I_n+D(t))P +(A(t)\psi(t) + I_n)b \mu_0^\rT\\
\label{Qdot}
    & =
    A(t)(P + Q(t)) + b \mu_0^\rT, 
    \qquad t \> 0, 
\end{align}
where (\ref{D}) is 
used along with the properties of the fundamental matrix $G$ and the related function $\psi$ from  (\ref{psi}). The mean vector (\ref{mu}) satisfies the ODE $\dot{\mu}(t) = A(t)\mu(t) + b$ with the same matrix $A(t)$ as the columns of the matrix $P+ Q(t)$ in (\ref{Qdot}). Hence, these two ODEs can be assembled into one ODE (\ref{zdot}) for the matrix (\ref{muQP}) with an appropriately augmented forcing term $\begin{bmatrix}
  b & b\mu_0^\rT
\end{bmatrix} = c$.  
\end{proof}

Since Theorem~\ref{th:Delz} relates the mean-square deviation functional $\Delta$ for the quantum memory  system to the $\mR^{n\x (n+1)}$-valued state $z$ of the auxiliary classical deterministic dynamical system (\ref{zdot}), any performance analysis or optimisation problem for $\Delta$ can be reformulated in terms of the classical system. Furthermore, due to the affine dependence of the matrix 
\begin{equation}
\label{AAA}
    A(t) = A_* + \cA \cdot U(t)
\end{equation}  
in (\ref{A_b}),  (\ref{A*})--(\ref{Ak})  on the control signal $U$,   the ODE (\ref{zdot}) 
is organised as a control-affine system \cite{S_1998}.  The affinity of the dependence on the control is inherited by the time derivative 
\begin{align}
\nonumber
    \dot{\Delta}(t)
    & = 
    \bra R, \dot{z}(t)\ket
    = 
    \bra R, A(t)z(t) + c\ket\\
\nonumber
    & =    
    \bra R, A_*z(t) + c\ket
    +
    \bra R, (\cA\cdot U(t))z(t)\ket\\
\nonumber
    & = 
    f(z(t))
    +
    \bra Rz(t)^\rT, \cA \cdot U(t)\ket\\    
\label{Deldot}
    & = 
    f(z(t))
    +
    g(z(t))^\rT U(t)
    =: h(z(t), U(t)),  
\end{align}
obtained from (\ref{DelRzd0}) by using (\ref{zdot}) along with (\ref{AAA}) and (\ref{R})--(\ref{ustarv}). Here, $f$, $g$ are affine and linear functions on $\mR^{n \x (n+1)}$ with values in $\mR$ and $\mR^r$, respectively, which are defined by 
\begin{equation}
\label{fg}
        f(v)
        := \bra R, A_*v + c\ket,
        \qquad
        g(v)
        := 
        (Rv^\rT)\star \cA,
        \qquad
        v \in \mR^{n\x (n+1)}. 
\end{equation}
The function $h: \mR^{n \x (n+1)}\x \mR^r \to \mR$ in (\ref{Deldot}) depends affinely on its second argument,  so that $g(v)= \d_u h(v,u)$ is the  gradient of $h(v,u)$ with respect to  the control variable $u \in \mR^r$.   Since $\Delta(t)\> 0$ for any $t\> 0$, then (\ref{Del0}) implies that $\dot{\Delta}(0)\> 0$ for any initial control $U_0 := U(0) \in \mR^r$. Therefore, if the vector $g(z_0)$ in (\ref{Deldot}), considered at $t=0$,  were nonzero, then there would exist $U_0$ such that $\dot{\Delta}(0) = f(z_0) + g(z_0)^\rT U_0 < 0$, thus contradicting the property that $\dot{\Delta}(0)\> 0$ for any $U_0\in \mR^r$. From this contradiction, it follows that $g(z_0) = 0$, and hence, 
\begin{equation}
\label{Deldot0}
  \dot{\Delta}(0) = f(z_0) \> 0. 
\end{equation}
In particular, $\dot{\Delta}(0)$ is independent of $U_0$. However, the second time derivative of $\Delta$  acquires dependence on the control signal $U$. More precisely, by assuming that $U$ is  continuously differentiable, the differentiation of (\ref{Deldot}) with respect to time yields
\begin{align}
\nonumber
    \ddot{\Delta}(t)
    = &  
    \bra
        R, A_* \dot{z}(t)
    \ket
    +
    g(\dot{z}(t))^\rT U(t) + g(z(t))\dot{U}(t)\\
\nonumber
    = &  
    \bra
        R, A_* (A_* z(t)+c)
    \ket
    +
    \bra
        R, A_* \wt{A}(t) z(t)
    \ket    \\
\nonumber
    & + 
    g(A_*z(t)+c)^\rT U(t) + g(\wt{A}(t)z(t))^\rT U(t)\\
\nonumber
     & + g(z(t))\dot{U}(t)    \\
\nonumber
    = &  
    \bra
        R, A_* (A_* z(t)+c)
    \ket
    +
    \bra
        A_*^\rT R z(t)^\rT, \cA \cdot U(t)
    \ket    \\
\nonumber
    & + 
    g(A_*z(t)+c)^\rT U(t) + g(\wt{A}(t)z(t))^\rT U(t)\\
\nonumber
     & + g(z(t))\dot{U}(t)    \\
\nonumber
    = &  
    \bra
        R, A_* (A_* z(t)+c)
    \ket\\
\nonumber
    & +
    ((A_*^\rT R z(t)^\rT + R(A_* z(t)+c)^\rT)\star \cA)^\rT U(t)\\
\nonumber
    & + 
    ((Rz(t)^\rT (\cA \cdot U(t))^\rT)\star \cA)^\rT U(t)\\
\label{Delddot}
     & + g(z(t))\dot{U}(t)    , 
\end{align} 
which is quadratic in $U(t)$ and affine in $\dot{U}(t)$. In particular, (\ref{Delddot}) shows that $\ddot{\Delta}(0)$ 
does not depend on $\dot{U}(0)$ as $g(z_0) = 0$.

\section{POINTWISE CONTROL OPTIMISATION}
\label{sec:lyap}

Recall that at any time $t>0$,  the quantity $\Delta(t)$ in (\ref{Del}) describes the mean-square  accuracy with which the quantum variables, comprising the process (\ref{FX}), approximately reproduce their initial conditions in (\ref{varphi0}). Its small values are beneficial for the quantum memory performance, which, due to the representation (\ref{DelRzd0}),  can be formulated  in terms of the auxiliary classical dynamical system (\ref{zdot}). Similarly to the gradient descent and speed gradient approaches to stabilisation of  control-affine systems (see, for example, \cite[Section~5.9]{S_1998} and \cite{AF_2021}), the control  $U$ can be used in order to steer the system (\ref{zdot}) towards small values of $\Delta$ by minimising $\dot{\Delta}(t)$ over $U(t)$ at every moment of time $t\>0$. However, in view of the affine dependence of $\dot{\Delta}(t)$ in (\ref{Deldot}) on $U(t)$, this minimisation, without constraints on $U$,  is not well-posed whenever $g(z(t))\ne 0$. Indeed, in this case, $\dot{\Delta}(t)$ can be assigned arbitrarily large negative values by an appropriate choice of sufficiently large controls $U(t) \in \mR^r$ forming an obtuse angle with the gradient vector $g(z(t))$.  Nevertheless, this minimisation problem can be regularised, and the control signal moderated,  by using an additional quadratic penalty $\|U\|_\Pi^2 := U^\rT \Pi U$, where $0\prec \Pi = \Pi^\rT \in \mR^{r\x r}$ is a given matrix. The resulting pointwise optimisation problem takes the form
\begin{equation}
\label{lyap} 
    \dot{\Delta}(t)+ \frac{1}{2}\|U(t)\|_\Pi^2 
    \longrightarrow \inf
\end{equation}
of a quadratic minimisation problem over     $U(t)\in \mR^r$ (with the $\frac{1}{2}$-factor incorporated for convenience) and is solved below. 
\begin{thm}
\label{th:opt1}
At any time $t\> 0$,  the problem (\ref{lyap}) has a unique solution
\begin{equation}
\label{Uhat}
    \wh{U}(t) = -\Pi^{-1} g(z(t)) ,  
\end{equation}
which is computed together with  
the corresponding value  
\begin{equation}
\label{Deldothat}
    \dot{\Delta}(t)  = f(z(t)) - \|g(z(t))\|_{\Pi^{-1}}^2    
\end{equation}
in terms of the functions $f$, $g$ from  (\ref{fg}). 
\hfill$\square$
\end{thm}
\begin{proof}
The solution (\ref{Uhat}) follows from (\ref{Deldot}) and the completion of the square in $g^\rT u + \frac{1}{2}\|u\|_\Pi^2 = \frac{1}{2} \|u + \Pi^{-1}g\|_\Pi^2- \frac{1}{2}\|g\|_{\Pi^{-1}}^2$ which achieves its minimum value $- \frac{1}{2}\|g\|_{\Pi^{-1}}^2$ only at $u =  -\Pi^{-1}g$. Accordingly, (\ref{Deldothat}) is obtained by substituting (\ref{Uhat}) into (\ref{Deldot}). 
\end{proof}

The control $\wh{U}(t)$ in  (\ref{Uhat}) is linear with respect to  the matrix $z(t)$  from (\ref{muQP}). Therefore, the corresponding matrix $A(t) = A_* + \cA \cdot \wh{U}(t)$ in (\ref{AAA}) depends affinely on $z(t)$, which gives rise to a quadratic nonlinearity in the ODE (\ref{zdot}): 
\begin{equation}
\label{zdot1}
    \dot{z}(t) = (A_* - \cA \cdot (\Pi^{-1} g(z(t))) )z(t) + c. 
\end{equation}
In view of (\ref{Del0}), integration of (\ref{Deldothat}) along the trajectory of (\ref{zdot1}) yields the mean-square deviation 
\begin{equation*}
\label{Delhat}
    \Delta(\tau)
    = 
    \int_0^\tau
    (f(z(t)) - \|g(z(t))\|_{\Pi^{-1}}^2)
    \rd t
\end{equation*}
at any time horizon $\tau>0$. The quadratic penalty on the control (\ref{Uhat}) in (\ref{lyap}) is taken into account in a similar fashion:
\begin{equation}
\label{Delhat1}
    \Delta(\tau)
    +
    \frac{1}{2}
    \int_0^\tau
    \|\wh{U}(t)\|_\Pi^2
    \rd t
    = 
    \int_0^\tau
    \Big(
        f(z(t)) - \frac{1}{2}\|g(z(t))\|_{\Pi^{-1}}^2
    \Big)
    \rd t. 
\end{equation}
Note that the nonlinearity of the ODE (\ref{zdot1}) makes the role of the matrix $\Pi$ for the memory performance nontrivial. Indeed, on the one hand, with $g(z(t))\ne 0$,   sufficiently small matrices $\Pi$  lead to $\dot{\Delta}(t)<0$ in (\ref{Deldothat}) on an appropriate set of moments of time $t>0$ (cf. (\ref{Deldot0})). For example, this holds whenever 
$$
    0 \prec 
    \Pi \prec  \frac{|g(z(t))|^2}{\max(0, f(z(t)))} I_r. 
$$
On the other hand, for small matrices $\Pi$, the auxiliary system dynamics (\ref{zdot1}) are substantially nonlinear, which complicates the relationship between negativity of $\dot{\Delta}$ on a subset of the time interval $[0,\tau]$ and the smallness of $\Delta(\tau)$.  

\section{FINITE-HORIZON OPTIMAL MEMORY CONTROL}
\label{sec:opt}

In comparison with the pointwise optimal control (\ref{lyap}), a related yet different approach to the quantum memory optimisation is provided by the minimisation of an integral cost functional
\begin{equation}
\label{Phi}
    \Phi(\tau)
    :=
  \Delta(\tau) + \frac{1}{2}\int_0^\tau \|U(t)\|_\Pi^2 \rd t
  = \int_0^\tau 
  \Big(
  \dot{\Delta}(t) + \frac{1}{2}\|U(t)\|_\Pi^2\Big) \rd t
  \longrightarrow \inf
\end{equation}
over continuous control signals $U \in C([0,\tau], \mR^r)$ with a given time horizon $\tau>0$. Here, as in (\ref{lyap}), the function $\Delta$ is associated by (\ref{DelRzd0}) with the $\mR^{n\x (n+1)}$-valued  state (\ref{muQP}) of the auxiliary dynamical system (\ref{zdot}). Application of dynamic programming to the problem (\ref{Phi}) uses the 
Bellman function 
\begin{equation}
\label{Psi}
    \Psi(t, v)
    := 
    \inf_{U\in C([t,\tau], \mR^r)}
    \Big(
    \Delta(\tau) 
    + 
    \frac{1}{2}\int_t^\tau \|U(s)\|_\Pi^2 \rd s
    \Big)
\end{equation}
for any $0 \< t\< \tau$ and $v \in \mR^{n\x (n+1)}$ (with the auxiliary system (\ref{zdot}) on $[t,\tau]$  initialised  at $z(t):= v$) and is described below. 

\begin{thm}
\label{th:opt2}
Suppose the Bellman function (\ref{Psi}) is continuously differentiable (including the continuous differentiablity of $\Psi$ with respect to its second argument in the sense of Frechet). Then the HJBE for it takes the form of the PDE 
\begin{equation}
\label{HJBE0}
    \d_t \Psi
    + 
        \bra
        \d_v \Psi,
        A_* v + c
        \ket
        -
        \frac{1}{2}
        \|(\d_v \Psi v^\rT)\star \cA \|_{\Pi^{-1}}^2 = 0  
\end{equation}
for all $0 \< t\< \tau$,  
with the terminal condition 
\begin{equation}
\label{Psitau}
    \Psi(\tau, v) = \bra R, v\ket + d,
    \qquad
    v \in \mR^{n\x (n+1)}.  
\end{equation}
The optimal control law, as a time-dependent function of the current state of the auxiliary system (\ref{zdot}),   is described by 
\begin{equation}
\label{uhat}
    \wh{u}(t,v) = - \Pi^{-1}((\d_v \Psi(t,v) v^\rT)\star \cA) 
\end{equation}
in terms of the $\mR^{n \x (n+1)}$-valued Frechet derivative $\d_v \Psi$.  
\hfill$\square$
\end{thm}
\begin{proof}
If $\Psi \in C^1([0,\tau]\x\mR^{n\x (n+1)}, \mR_+)$,  then it satisfies the HJBE whose form is specified 
by the dependence of the right-hand of the ODE (\ref{zdot}) on the control through (\ref{AAA}) and the control penalty in (\ref{Phi})  as
\begin{align}
\nonumber
    -\d_t \Psi
    &= 
    \min_{u \in \mR^r}
    \Big(
        \bra
        \d_v \Psi,
        (A_* + \cA\cdot u) v + c
    \ket  + 
    \frac{1}{2}
    \|u \|_\Pi^2
    \Big)\\
\nonumber
    &=  
        \bra
        \d_v \Psi,
        A_* v + c
        \ket
        +
    \min_{u \in \mR^r}
    \Big(
        ((\d_v \Psi v^\rT)
        \star \cA)^\rT 
        u
        + 
    \frac{1}{2}
    \|u \|_\Pi^2
    \Big)\\    
\label{HJBE+}
    &=  
        \bra
        \d_v \Psi,
        A_* v + c
        \ket
        -
        \frac{1}{2}    
        \|(\d_v \Psi v^\rT)
        \star \cA\|_{\Pi^{-1}}^2  , 
\end{align}
which is equivalent to the quadratically nonlinear PDE (\ref{HJBE0}). The minimum in (\ref{HJBE+}) is found as in the proof of Theorem~\ref{th:opt1} for the pointwise optimisation problem (\ref{lyap}) and, similarly to (\ref{Uhat}), is achieved at the unique point (\ref{uhat}). It now remains to note that the terminal value of $\Psi$ in 
(\ref{Psitau}) reproduces $\Delta(\tau)$ from (\ref{DelRzd0}) as a function of the terminal state $z(\tau)$ of the system (\ref{zdot}).
\end{proof}

In accordance with (\ref{AAA}), under the optimal control law $U(t) = \wh{u}(t, z(t))$ from (\ref{uhat}), the auxiliary dynamical system  (\ref{zdot}) is governed by a nonlinear ODE
\begin{equation}
\label{zdot2}
  \dot{z}(t) = (A_* + \cA \cdot \wh{u}(t,z(t)))z(t) + c. 
\end{equation}
The right-hand side of (\ref{HJBE+}), considered on the trajectory $v:= z(t)$ of (\ref{zdot2}),  coincides with the minimum value of the Pontryagin control Hamiltonian \cite{PBGM_1962,SW_1997}
\begin{align*}
    \cH
    & :=
    \Big(
        \bra
        \d_v \Psi,
        (A_* + \cA\cdot u) v + c
    \ket  + 
    \frac{1}{2}
    \|u \|_\Pi^2
    \Big)\Big|_{v=z, u =\wh{u}}\\
    & =
        \bra
        \d_v \Psi,
        A_*z + c
    \ket  
    -
    \frac{1}{2} \|\wh{u}\|_\Pi^2
\end{align*}
and remains constant over the time interval $[0,\tau]$.  The corresponding  minimum value of the cost functional (\ref{Phi}) is given by 
\begin{equation}
\label{PhiPsi}
    \min \Phi(\tau) = \Psi(0,z_0),
    \qquad
    z_0 \in \cZ, 
\end{equation}
in terms of (\ref{Psi}) and the initial condition $z_0$ of the auxiliary system (\ref{zdot}) from the set (\ref{cZ}). The quantity (\ref{Delhat1}) provides an upper bound for $\min \Phi(\tau)$ in (\ref{PhiPsi}), so that the pointwise optimal control  (\ref{Uhat}) can be regarded as a suboptimal solution to the problem (\ref{Phi}).   The solution of the terminal value problem for the HJBE (\ref{HJBE0}) can be approached by using a penalty matrix 
\begin{equation}
\label{Pieps}
    \Pi_\eps:= \frac{1}{2\eps}\Gamma
\end{equation}
whose ``shape'' and ``scale'' are specified by a given matrix $0 \prec \Gamma = \Gamma^\rT \in \mR^{r\x r}$ and a small scalar parameter $\eps>0$, respectively.  In the resulting one-parameter family of HJBEs
\begin{equation}
\label{HJBE2}
    \d_t \Psi_\eps
    + 
        \bra
        \d_v \Psi_\eps,
        A_* v + c
        \ket
        =
        \eps
        \|(\d_v \Psi_\eps v^\rT)\star \cA \|_{\Gamma^{-1}}^2 
\end{equation}
with the same terminal condition (\ref{Psitau}),  the right-hand side plays the role of a nonlinear perturbation of a first-order homogenous linear PDE $    \d_t \Psi_0 
    + 
        \bra
        \d_v \Psi_0,
        A_* v + c
        \ket = 0
$. The latter corresponds to the case of zero control $U=0$ in (\ref{zdot}), (\ref{AAA}) and is solved under the condition (\ref{Psitau}) for $\Psi_0$  by the method of characteristics as
\begin{equation}
\label{Psi0}
    \Psi_0(t,v)
    =
    \Bra
        R, 
        \re^{(\tau-t)A_*}v
        +
        \int_0^{\tau-t}
        \re^{sA_*}
        \rd sc
    \Ket + d,  
\end{equation}
where $\int_0^T
        \re^{sA_*}
        \rd s = (\re^{TA_*}-I_n)A_*^{-1}$ when $\det A_* \ne 0$.  
A solution of (\ref{HJBE2})  can then be obtained in the form of an asymptotic expansion 
\begin{equation}
\label{Psieps}
    \Psi_\eps(t,v) = \sum_{k=0}^{+\infty}\eps^k \Psi^{(k)}(t,v), 
    \quad
    0 \< t\< \tau, \
    v \in \mR^{n\x (n+1)},   
\end{equation}
where $\Psi^{(0)}:= \Psi_0$ is given by (\ref{Psi0}), and the subsequent terms are computed recursively by solving a nonhomogeneous  first-order linear PDE
\begin{equation}
\label{HJBE3}
    \d_t \Psi^{(k)}
    + 
        \bra
        \d_v \Psi^{(k)},
        A_* v + c
        \ket
        =
        \sum_{\ell=0}^{k-1}
        ((\d_v \Psi^{(\ell)} v^\rT)\star \cA)^\rT \Gamma^{-1}
        ((\d_v \Psi^{(k-\ell-1)} v^\rT)\star \cA) 
\end{equation}
with the terminal condition $\Psi^{(k)}(\tau,\cdot) = 0$ for all $k = 1, 2, 3, \ldots$.  Substitution  of (\ref{Pieps}), (\ref{Psieps}) into (\ref{uhat}) yields an asymptotic expansion for the corresponding optimal control law:
\begin{align}
\label{ueps}
    \wh{u}_\eps(t,v)
    & =
    \sum_{k=0}^{+\infty}
    \eps^{k+1} \wh{u}^{(k)}(t,v),\\
\label{uk}
    \wh{u}^{(k)}(t,v)
    & :=  
    - 2\Gamma^{-1}((\d_v \Psi^{(k)}(t,v) v^\rT)\star \cA). 
\end{align}
In particular, since (\ref{Psi0}) implies that  $\d_v \Psi_0(t,v) = \re^{(\tau-t)A_*^\rT} R$, so that 
\begin{equation}
\label{u0}
      \wh{u}^{(0)}(t,v)
     =  
    - 2\Gamma^{-1}((\re^{(\tau-t)A_*^\rT} R v^\rT)\star \cA), 
\end{equation}
then the PDE (\ref{HJBE3}) for the function $\Psi^{(1)}$ acquires the form 
\begin{equation}
\label{Psi1}
    \d_t \Psi^{(1)}(t,v)
     + 
        \bra
        \d_v \Psi^{(1)}(t,v),
        A_* v + c
        \ket
     = \|(\re^{(\tau-t)A_*^\rT} R v^\rT)\star \cA\|_{\Gamma^{-1}}^2.  
\end{equation}
Its solution $\Psi^{(1)}(t,v)$ and the corresponding control $u^{(1)}(t,v)$ in (\ref{uk}) are  quadratic functions of $v$ with time-varying coefficients which will be calculated elsewhere.  
The functions $\Psi_0 + \eps \Psi^{(1)}$ and $\eps \wh{u}^{(0)} + \eps^2 \wh{u}^{(1)}$, resulting from (\ref{Psi0}), (\ref{uk})--(\ref{Psi1}),  provide the first-order truncations of (\ref{Psieps}), (\ref{ueps}), respectively.

\section{CONCLUSION}
\label{sec:conc}

We have considered performance optimisation for a finite-level quantum memory system which is controlled through the time-varying  parameters of its Hamiltonian. The mean-square deviation of quantum system variables of interest from their initial values (to be retained by the system) has been represented in terms of an auxiliary classical dynamical system in a matrix space, which the control enters in a multiplicative fashion.  We have discussed the pointwise minimisation of its time derivative with a quadratic penalty  on the control, which can be regarded as a regularised version of the  gradient descent and speed gradient approaches to control-affine systems. The resulting control law is linear with respect to the state of the  auxiliary system, thus making the dynamics of the latter quadratically nonlinear. An alternative yet related optimal control design using a finite-horizon integral cost functional has also been considered. We have applied dynamic programming  and obtained a quadratically nonlinear HJBE with a recursively computed asymptotic expansion for its solution.

\end{document}